\documentclass[11pt,dvipdfm]{article}
\usepackage{amsfonts}
\usepackage{graphicx}
\usepackage{amsfonts, amsmath, amssymb, latexsym, epsfig}
\usepackage{mathrsfs}
\newcommand{\blst}{\begin{trivlist}}
\newcommand{\elst}{\end{trivlist}}

\newtheorem{thm}{Theorem}[section]
\newtheorem{prop}[thm]{Proposition}
\newtheorem{cor}[thm]{Corollary}
\newtheorem{lem}[thm]{Lemma}
\newtheorem{conj}[thm]{Conjecture}
\newtheorem{exa}[thm]{Example}
\newtheorem{defn}[thm]{Definition}
\newtheorem{rem}[thm]{Remark}

\newcommand{\ben}{\begin{enumerate}}
\newcommand{\een}{\end{enumerate}}
\newcommand{\ble}{\begin{lem}}
\newcommand{\ele}{\end{lem}}
\newcommand{\bth}{\begin{thm}}
\renewcommand{\eth}{\end{thm}}
\newcommand{\bpr}{\begin{prop}}
\newcommand{\epr}{\end{prop}}
\newcommand{\bco}{\begin{cor}}
\newcommand{\eco}{\end{cor}}
\newcommand{\bcon}{\begin{conj}}
\newcommand{\econ}{\end{conj}}
\newcommand{\bde}{\begin{defn}}
\newcommand{\ede}{\end{defn}}
\newcommand{\bex}{\begin{exa}}
\newcommand{\eex}{\end{exa}}
\newcommand{\barr}{\begin{array}}
\newcommand{\earr}{\end{array}}
\newcommand{\btab}{\begin{tabular}}
\newcommand{\etab}{\end{tabular}}
\newcommand{\beq}{\begin{equation}}
\newcommand{\eeq}{\end{equation}}

\newcommand{\bea}{\begin{eqnarray*}}
\newcommand{\eea}{\end{eqnarray*}}
\newcommand{\beaa}{\begin{eqnarray}}
\newcommand{\eeaa}{\end{eqnarray}}
\newcommand{\bce}{\begin{center}}
\newcommand{\ece}{\end{center}}
\newcommand{\bpi}{\begin{picture}}
\newcommand{\epi}{\end{picture}}
\newcommand{\bfi}{\begin{figure} \begin{center}}
\newcommand{\efi}{\end{center} \end{figure}}

\newcommand{\bsl}{\begin{slide}{}}
\newcommand{\esl}{\end{slide}}

{}
{}
{}
{}
{}
{}
{}
{}
{}
{}
{}
{}
{}
{}
{}
{}
{}
{}
{}
{}
{}
{}
{}
{}
{}
{}
{}
{}
\newenvironment{proof}{
\par
\noindent {\bf Proof.}\rm}{\mbox{}\hfill\rule{0.5em}{0.809em}\par}
\begin{document}
\title{Generalizations of The Chung-Feller Theorem II}

\author{
Jun Ma$^{a,}$\thanks{Email address of the corresponding author:
majun@math.sinica.edu.tw}
 \and Yeong-Nan Yeh $^{b,}$\thanks{Partially supported by NSC 96-2115-M-001-005}}

\date{}
\maketitle \vspace{-1cm} \bce \footnotesize
 $^{a,b}$ Institute of Mathematics, Academia Sinica, Taipei, Taiwan\\
\ece

\thispagestyle{empty}\vspace*{.4cm}

\begin{abstract} The classical Chung-Feller theorem \cite{CF} tells us that
the number of Dyck paths of length $n$ with $m$ flaws  is the $n$-th
Catalan number and independent on $m$. L. Shapiro \cite{S} found the
Chung-Feller properties for the Motzkin paths. Mohanty's book
\cite{SGM} devotes an entire section to exploring Chung-Feller
theorem. Many Chung-Feller theorems are consequences of the results
in \cite{SGM}. In this paper, we consider the $(n,m)$-lattice paths.
We study two parameters for an $(n,m)$-lattice path: the
non-positive length and the rightmost minimum length. We obtain the
Chung-Feller theorems of the $(n,m)$-lattice path on these two
parameters by bijection methods.  We are more interested in the
pointed $(n,m)$-lattice paths. We investigate two parameters for an
pointed $(n,m)$-lattice path: the pointed non-positive length and
the pointed rightmost minimum length. We generalize the results in
\cite{SGM}.  Using the main results in this paper, we may find the
Chung-Feller theorems of many different lattice paths.
\end{abstract}


\section{Introduction}
Let $\mathbb{Z}$ denote the set of the integers and
$[n]:=\{1,2,\ldots,n\}$. We consider $n$-Dyck paths in the plane
$\mathbb{Z}\times \mathbb{Z}$ using {\it up $(1,1)$} and {\it down
$(1,-1)$} steps that go from the origin to the point (2n,0). We say
$n$ the \emph{semilength} because there are $2n$ steps.  An $n$-{\it
flawed} path is an $n$-Dyck path that contains some steps under the
$x$-axis. The number of $n$-Dyck path that never pass below the
$x$-axis is the $n$-th Catalan number
$c_n=\frac{1}{n+1}{2n\choose{n}}$. Such paths are called the {\it
Catalan paths of length $n$}. A Dyck path is called a $(n,r)$-{\it
flawed} path if it contains $r$ up steps under the $x$-axis and its
semilength is $n$. Clearly, $0\leq r\leq n$. The classical
Chung-Feller
 theorem \cite{CF} says that the number of the $(n,r)$-{
flawed} paths is equal to $c_n$ and independent on $r$.

The classical Chung-Feller Theorem were proved by MacMahon
\cite{MacMahon}. Chung and Feller reproved this theorem by using
analytic method in \cite{CF}. T.V.Narayana \cite{N} showed the
Chung-Feller Theorem by combinatorial methods.  S. P. Eu et al.
\cite{EFY2} proved the Chung-Feller Theorem by using the Taylor
expansions of generating functions and gave a refinement of this
theorem. In \cite{EFY}, they gave a strengthening of the
Chung-Feller Theorem and a weighted version for Schr\"{o}der paths.
Y.M. Chen \cite{C} revisited the Chung-Feller Theorem by
establishing a bijection.

Mohanty's book \cite{SGM} devotes an entire section to exploring
Chung-Feller theorem. We state the result from \cite{SGM} as the
following lemma.

\begin{lem}\label{bookchung}\cite{SGM} Given a positive integer $n$, let $Y=(y_1,\ldots,y_{n+1})$ be
a sequence of integers with $1-n\leq y_i\leq 1$ for all $i\in[n+1]$
such that $\sum\limits_{i=1}^{n+1}y_i=1$. Furthermore, let
${E}(Y)=|\{i\mid \sum\limits_{j=1}^iy_j\leq 0\}|$. Let $Y_i$ be the
$i$-th cyclic permutation of $Y$ (i.e.,
$Y_i=(y_i,y_{i+1},\ldots,y_{n+i+1})$ with $y_{n+r+1}=y_r$). Then
there exists a permutation $i_1,\ldots, i_{n+1}$ on the set $[n+1]$
such that ${E}(Y_{i_1})>{E}(Y_{i_2})>\cdots
>{E}(Y_{i_{n+1}})$.\end{lem}

Many Chung-Feller theorems are consequences of lemma
\ref{bookchung}. First, let $\phi$ be a mapping from $\mathbb{Z}$ to
$\mathbb{P}$, where $\mathbb{P}$ is a set of all the positive
integers. Let the sequence $Y=(y_1,\ldots, y_{n+1})$ satisfy the
conditions in Lemma \ref{bookchung}. Using $(\phi(y_i),y_i)$ steps,
we can obtain a lattice path
$P(Y)=(\phi(y_1),y_1)(\phi(y_2),y_2)\ldots (\phi(y_{n+1}),y_{n+1})$
in the plane $\mathbb{Z}\times\mathbb{Z}$ that go from the origin to
the point $(\sum\limits_{i=1}^{n+1}\phi(y_i),1)$.  Using Lemma
\ref{bookchung}, we will derive the classical Chung-Feller theorem
for Dyck paths if we let $y_i\in\{1,-1\}$ and set $\phi(y)=1$ for
all $y\in\mathbb{Z}$; we will derive the Chung-Feller theorem for
Schr\"{o}der paths if we let $y_i\in\{1,0,-1\}$ and set $\phi(0)=2$
and $\phi(y)=1$ for $y\neq 0$; we will derive the Chung-Feller
theorem for Motzkin paths if we let $y_i\in\{1,0,-1\}$ and set
$\phi(0)=1$ and $\phi(y)=1$ for $y\neq 0$ and so on.

How to derive the Chung-Feller theorem for lattice paths in the
plane $\mathbb{Z}\times\mathbb{Z}$ using $(1,-1)$, $(1,1)$, $(1,0)$,
$(2,0)$ steps? For answering this problem, the authors of this paper
\cite{Ma} proved the Chung-Feller theorems for three classes of
lattice paths by using the method of the generating functions. It is
interesting that these Chung-Feller theorems can't be derivable as a
special case from lemma \ref{bookchung}. This implies that we may
generalize the results of Lemma \ref{bookchung}.

In this paper, first we give the definition of the $(n,m)$-lattice
paths. We consider two parameters for an $(n,m)$-lattice path: the
non-positive length and the rightmost minimum length. Using
bijection methods, we obtain the Chung-Feller theorems of the
$(n,m)$-lattice path on these two parameters. Then we study the
pointed $(n,m)$-lattice paths. We investigate two parameters for an
pointed $(n,m)$-lattice path: the pointed non-positive length and
the pointed rightmost minimum length. We give generalizations of the
results in \cite{SGM} and prove the Chung-Feller theorems of the
pointed $(n,m)$-lattice path on these two parameters. Finally, using
the main theorems of this paper, we may find the Chung-Feller
theorems of many different $(n,m)$-lattice paths.

This paper is organized as follows. In Section $2$, we focus on the
$(n,m)$-lattice paths. Using bijection methods, we obtain the
Chung-Feller theorems of the $(n,m)$-lattice path. In Section $3$,
we study the pointed $(n,m)$-lattice paths and give generalizations
of the results in \cite{SGM}. In Section $4$, using the main
theorems of this paper, we find the Chung-Feller theorems of many
different $(n,m)$-lattice paths.

\section{The $(n,m)$-lattice paths}
Throughout the paper, we always let $n$ and $m$ be two positive
integers with $m\geq n+1$. In this section, we will consider the
$(n,m)$-lattice paths. We will define two parameters for an
$(n,m)$-lattice path: the non-positive length and the rightmost
minimum length. Using bijection methods, we will obtain the
Chung-Feller theorems of the $(n,m)$-lattice path on these two
parameters. First, we give the definition of the $(n,m)$-lattice
paths as follows.
\begin{defn} An $(n,m)$-lattice paths $P$ is a sequence of the vectors $(x_1,y_1)(x_2,y_2)\ldots (x_{n+1},y_{n+1})$  in
$\mathbb{Z}^2$ such that:

(1) $1-n\leq y_i\leq 1$ and $\sum\limits_{i=1}^{n+1}y_i=1$

 (2) $1\leq x_i\leq m-1$ and $\sum\limits_{i=1}^{n+1}x_i=m$.
\end{defn}
$(x_i,y_i)$ is called the steps of $P$ for any $i\in [n+1]$. Since
$P$ can be viewed as a path from the origin to $(m,1)$ in the plane
$\mathbb{Z}\times \mathbb{Z}$ and has $n+1$ steps, we say that $P$
is of order $n+1$ and length $m$.
\subsection{The non-positive length of an $(n,m)$-lattice paths}
Given an $(n,m)$-lattice path $P=(x_1,y_1)(x_2,y_2)\ldots
(x_{n+1},y_{n+1})$, we let $NP(P)=\{i\mid \sum\limits_{j=1}^iy_j\leq
0\}$ and $NPL(P)=\sum\limits_{i\in NP(P)}x_i$. Clearly, $0\leq
NPL(P)\leq m-x_{n+1}\leq m-1$ since $n+1\neq NP(P)$. We say that
$NPL(P)$ is the {\it non-positive length} of the $(n,m)$-lattice
path $P$. Moreover, we define a linear order $<_P$ on the set
$[n+1]$ by the following rules:

for any $i,j\in[n+1]$, $i<_P j$ if either (1)
$\sum\limits_{k=1}^iy_k<\sum\limits_{k=1}^jy_k$ or (2)
$\sum\limits_{k=1}^iy_k=\sum\limits_{k=1}^jy_k$ and $i>j$.

The sequence formed by writing $[n+1]$ in the increasing order with
respect to $<_P$ is denoted by
$\pi_P=(\pi_P(1),\pi_P(2),\ldots,\pi_P(n+1))$.

\begin{exa}\label{examdyck111} Let $n=8$ and $m=11$. We draw an $(8,11)$-lattice path
$$P=(1,1)(1,-2)(2,1)(1,1)(1,-1)(1,-1)(1,1)(1,1)(2,0)$$ as follows.
\begin{center}
\includegraphics[width=10cm]{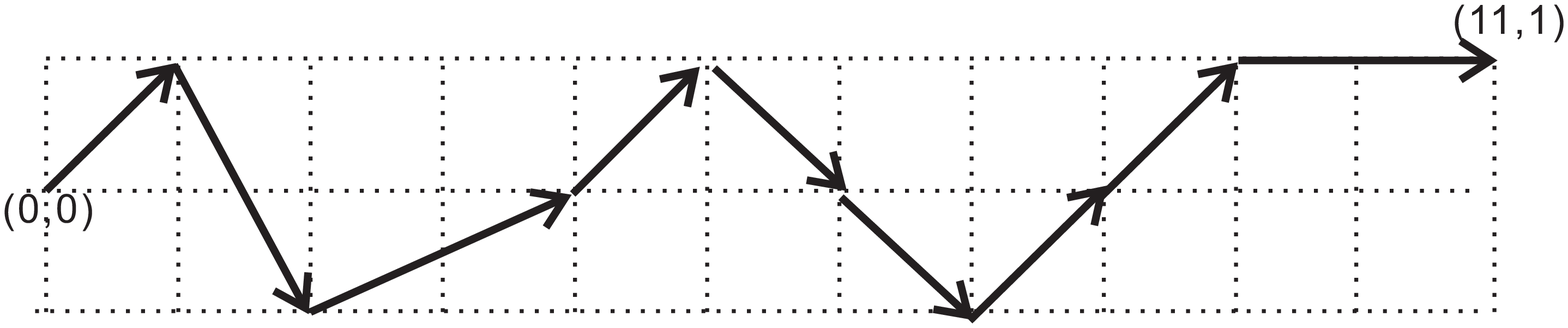}
\end{center}Then $NP(P)=\{2,3,5,6,7\}$, $NPL(P)=6$ and
$\pi_P=(6,2,7,5,3,9,8,4,1)$.
\end{exa}

We use $\mathcal{L}_{n,m,r}$ to denote the set of all the
$(n,m)$-lattice paths $P$ such that $NPL(P)=r$. In particularly, we
use  $\tilde{\mathcal{L}}_{n,m,0}$ to denote the set of all the
lattice paths $P=(x_1,y_1)(x_2,y_2)\ldots (x_{n+1},y_{n+1})$ in the
set $\mathcal{L}_{n,m,0}$ such that $x_{n+1}=1$. Clearly,
$\tilde{\mathcal{L}}_{n,m,0}\subset{\mathcal{L}}_{n,m,0}$.

\begin{lem}\label{enumeration}\item {(1) The number of the
$(n,m)$-lattice paths $P$ such that $NPL(P)=0$ is equal to
${m-1\choose{n}}c_n$;\\
(2) The number of the $(n,m)$-lattice paths
$P=(x_1,y_1)(x_2,y_2)\ldots (x_{n+1},y_{n+1})$ such that $NPL(P)=0$
and $x_{n+1}=1$ is equal to ${m-2\choose{n-1}}c_n$.}
\end{lem}
\begin{proof} (1) It is well known that the number of the solutions
of the equation $\sum\limits_{i=1}^{n+1}y_i=1$ such that $1-n\leq
y_i\leq 1$ and $NP(P)=\emptyset$ is $c_n$ and the number of the
solutions of the equation $\sum\limits_{i=1}^{n+1}x_i=m$ in positive
integers is ${m-1\choose{n}}$. Hence, The number of the
$(n,m)$-lattice paths $P$ such that $NPL(P)=0$ is equal to
${m-1\choose{n}}c_n$.

(2) Note that the number of the solutions of the equation
$\sum\limits_{i=1}^{n}x_i=m-1$ in positive integers is
${m-2\choose{n-1}}$. We immediately obtain that the number of the
$(n,m)$-lattice paths $P$ such that $NPL(P)=0$ and $x_{n+1}$ is
equal to ${m-2\choose{n-1}}c_n$.
\end{proof}

\begin{lem}\label{r>1} There is a bijection $\Phi$ from $\mathcal{L}_{n,m,r}$
to $\mathcal{L}_{n,m,r+1}$ for any $1\leq r\leq m-2$.
\end{lem}
\begin{proof} Let $P=(x_1,y_1)(x_2,y_2)\ldots
(x_{n+1},y_{n+1})\in \mathcal{L}_{n,m,r}$. Consider the sequence
$\pi_P$.
 Suppose $\pi_P(k)=n+1$ for some $k$. Since $r\geq 1$, we have $k\geq
 2$. We discuss the following two cases.

 {\it Case I.} $k\leq n$

 If $x_{n+1}=1$, then let $i=\pi_P(k+1)$ and $$\Phi(P)=(x_{i+1},y_{i+1})\ldots
(x_{n+1},y_{n+1})(x_{1},y_{1})\ldots (x_{i},y_{i}).$$

 If $x_{n+1}\geq 2$, then let $i=\pi_P(k-1)$ and $$\Phi(P)=(x_{1},y_{1})\ldots
(x_{i}+1,y_{i})\ldots (x_{n+1}-1,y_{n+1}).$$

{\it Case II.} $k=n+1$

Note that $x_{n+1}\geq 2$ since $r\leq m-2$. We let $i=\pi_P(n)$ and
 $$\Phi(P)=(x_{1},y_{1})\ldots (x_{i}+1,y_{i})\ldots
(x_{n+1}-1,y_{n+1}).$$  It is easy to see that $\Phi(P)\in
\mathcal{L}_{n,m,r+1}$ for Cases I and II.

For proving that $\Phi$ is a bijection, we describe the inverse of
$\Phi$ as follows.

Let $P=(x_1,y_1)(x_2,y_2)\ldots (x_{n+1},y_{n+1})\in
\mathcal{L}_{n,m,r+1}$, where $1\leq r\leq m-2$. Suppose
$\pi_P(k)=n+1$ for some $k$. Let $i=\pi_P(k-1)$. If $x_i=1$, then
let $$\Phi^{-1}(P)=(x_{i+1},y_{i+1})\ldots
(x_{n+1},y_{n+1})(x_{1},y_{1})\ldots (x_{i},y_{i});$$otherwise,
let$$\Phi^{-1}(P)=(x_{1},y_{1})\ldots (x_{i}-1,y_{i})\ldots
(x_{n+1}+1,y_{n+1}).$$ This complete the proof.
\end{proof}
\begin{exa} Let $n=3$ and $m=5$. We draw $(3,5)$-lattice paths
\begin{eqnarray*}P_1=(1,1)(1,1)(1,-2)(2,1)&P_2=(1,1)(1,1)(2,-2)(1,1)\\
P_3=(1,1)(2,-2)(1,1)(1,1)&P_4=(2,-2)(1,1)(1,1)(1,1)\end{eqnarray*}
as follows.
\begin{center}
\includegraphics[width=14cm]{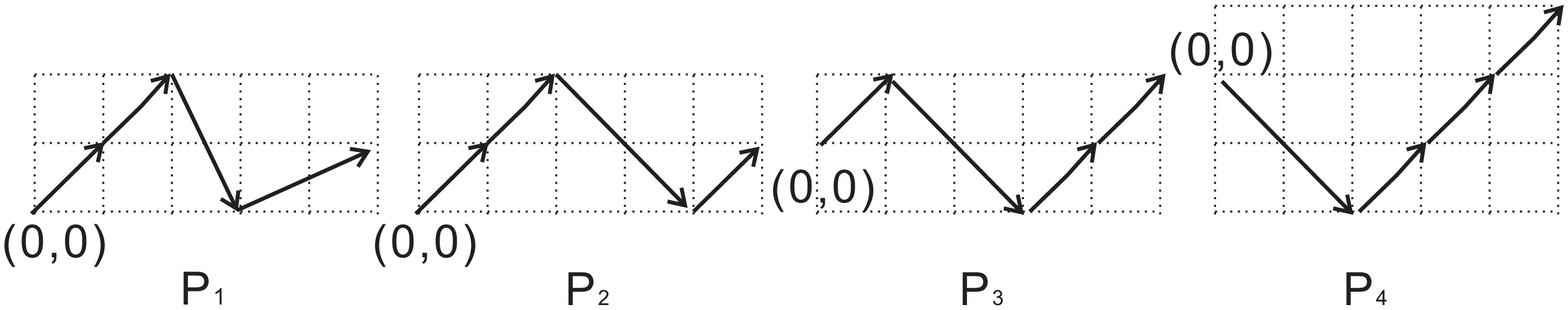}
\end{center}
We have $\Phi(P_i)=P_{i+1}$ and $NPL(P_i)=i$.
\end{exa}

\begin{lem}\label{r=0}There is a bijection from $\tilde{\mathcal{L}}_{n,m,0}$
to $\mathcal{L}_{n,m,1}$.
\end{lem}
\begin{proof} Let $P=(x_1,y_1)(x_2,y_2)\ldots (x_{n+1},y_{n+1})\in
\tilde{\mathcal{L}}_{n,m,0}$. Consider the sequence $\pi_P$. Note
that $\pi_P(1)=n+1$ for any $P\in\mathcal{L}_{n,m,0}$. So, let
$i=\pi_{P}(2)$. Let the mapping $\Phi$ be defined as that in Lemma
\ref{r>1}, i.e., $\Phi(P)=(x_{i+1},y_{i+1})\ldots
(x_{n+1},y_{n+1})(x_{1},y_{1})\ldots (x_{i},y_{i}).$  Then
$\Phi(P)\in\mathcal{L}_{n,m,1}$. Conversely, for any
$P=(x_1,y_1)(x_2,y_2)\ldots (x_{n+1},y_{n+1})\in
{\mathcal{L}}_{n,m,1}$, we have $\pi_P(2)=n+1$. Suppose
$\pi_P(1)=i$, then $x_i=1$. This tells us that $\Phi$ is a bijection
from $\tilde{\mathcal{L}}_{n,m,0}$ to $\mathcal{L}_{n,m,1}$.
\end{proof}

\begin{thm} For any $1\leq r\leq m-1$, the number of the
$(n,m)$-lattice paths $P$ such that $NPL(P)=r$ is equal to the
number of the $(n,m)$-lattice paths $P=(x_1,y_1)(x_2,y_2)\ldots
(x_{n+1},y_{n+1})$ such that $NPL(P)=0$ and $x_{n+1}=1$ and
independent on $r$.
\end{thm}
\begin{proof} Combining Lemmas \ref{r>1} and \ref{r=0}, we
immediately obtain the results as desired.
\end{proof}
\subsection{The rightmost minimum length of an $(n,m)$-lattice paths}
Given a $(n,m)$-lattice path $P=(x_1,y_1)(x_2,y_2)\ldots
(x_{n+1},y_{n+1})$, we let $a_0=0$, $b_0=0$,
$a_i=\sum\limits_{j=1}^iy_j$ and $b_i=\sum\limits_{j=1}^ix_j$ for
$i\geq 1$. Then the $(n,m)$-lattice path $P$ can be viewed as a
sequence of the points in the plane $\mathbb{Z}\times\mathbb{Z}$
$$(b_0,a_0),(b_1,a_1),\ldots,(b_{n+1},a_{n+1}).$$
A {\it minimum point} of the path $P$ is a point $(b_i,a_i)$ such
that $a_i\leq a_j$ for all $j\neq i$. A {\it rightmost minimum
point} is a minimum point $(b_i,a_i)$ such that the point is the
rightmost one among all the minimum points. If $(b_i,a_i)$ is {the
minimum point} of the path $P$, we call $b_i$ the {\it rightmost
minimum length} of the $(n,m)$-lattice paths $P$, denoted by
$RML(P)$.
\begin{exa} We consider the path $P$ in Example \ref{examdyck111}.
The point $(7,-1)$ is the rightmost minimum point and $RML(P)=7$.
\end{exa}
We use $\mathcal{M}_{n,m,r}$ to denote the set of all the
$(n,m)$-lattice paths $P$ such that $RML(P)=r$.

\begin{lem}\label{mr>1} There is a bijection $\Psi$ from $\mathcal{M}_{n,m,r}$
to $\mathcal{M}_{n,m,r+1}$ for any $1\leq r\leq m-2$.
\end{lem}
\begin{proof} Let $P=(x_1,y_1)(x_2,y_2)\ldots
(x_{n+1},y_{n+1})\in \mathcal{M}_{n,m,r}$. If $x_{n+1}=1$, we let
$$\Psi(P)=(x_{n+1},y_{n+1})(x_{1},y_{1})\ldots (x_{n},y_{n});$$
otherwise let $$\Psi(P)=(x_{1}+1,y_{1})(x_2,y_2)\ldots
(x_n,y_n)(x_{n+1}-1,y_{n+1}).$$ It is easy to see that $\Psi(P)\in
\mathcal{M}_{n,m,r+1}$.

For proving that $\Phi$ is a bijection, we describe the inverse of
$\Phi$ as follows.

If $x_{1}=1$, we let
$$\Psi(P)=(x_{2},y_{2})(x_{3},y_{3})\ldots (x_{n+1},y_{n+1})(x_1,y_1);$$
otherwise let $$\Psi(P)=(x_{1}-1,y_{1})(x_2,y_2)\ldots
(x_n,y_n)(x_{n+1}+1,y_{n+1}).$$ This complete the proof.
\end{proof}
\begin{exa} Let $n=3$ and $m=5$. We draw $(3,5)$-lattice paths
\begin{eqnarray*}P_1=(1,-2)(2,1)(1,1)(1,1)&P_2=(1,1)(1,-2)(2,1)(1,1)\\
P_3=(1,1)(1,1)(1,-2)(2,1)&P_4=(2,1)(1,1)(1,-2)(1,1)\end{eqnarray*}
as follows.
\begin{center}
\includegraphics[width=14cm]{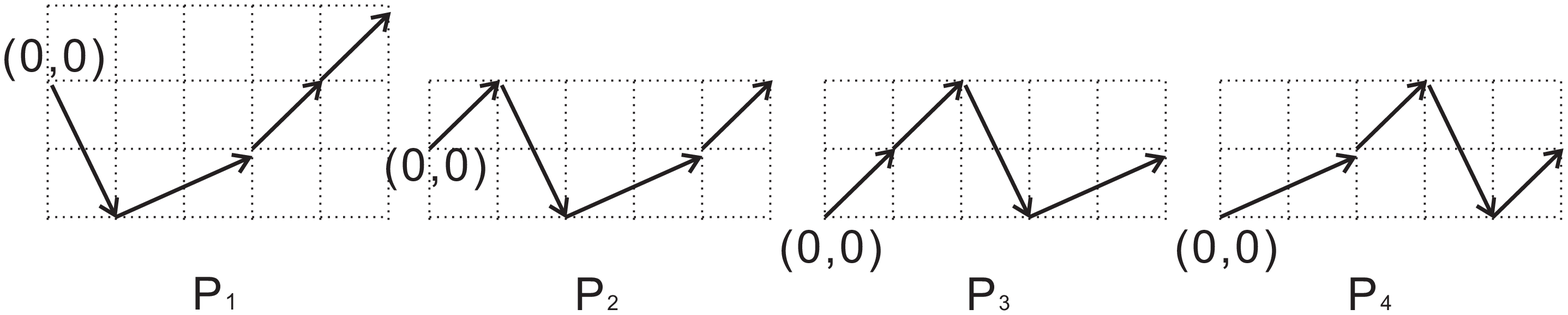}
\end{center}
We have $\Psi(P_i)=P_{i+1}$ and $RML(P_i)=i$.
\end{exa}
Note that $NPL(P)=0$ if and only if $RML(P)=0$ for any
$(n,m)$-lattice path. Recall that $\tilde{\mathcal{L}}_{n,m,0}$ is
the set of all the lattice paths $P=(x_1,y_1)(x_2,y_2)\ldots
(x_{n+1},y_{n+1})$ in the set $\mathcal{L}_{n,m,0}$ such that
$x_{n+1}=1$. Hence, also $\tilde{\mathcal{L}}_{n,m,0}$ is the set of
all the lattice paths $P=(x_1,y_1)(x_2,y_2)\ldots (x_{n+1},y_{n+1})$
in the set $\mathcal{M}_{n,m,0}$ such that $x_{n+1}=1$.
\begin{lem}\label{mr=0}There is a bijection from $\tilde{\mathcal{L}}_{n,m,0}$
to $\mathcal{M}_{n,m,1}$.
\end{lem}
\begin{proof} Let $P=(x_1,y_1)(x_2,y_2)\ldots
(x_{n+1},y_{n+1})\in \tilde{\mathcal{L}}_{n,m,0}$. Then $x_{n+1}=1$
and $y_{n+1}\leq 0$. We let
$$\Psi(P)=(x_{n+1},y_{n+1})(x_{1},y_{1})\ldots (x_{n},y_{n}).$$
Clearly, $\Psi(P)\in \mathcal{M}_{n,m,1}$.

Conversely, let $P=(x_1,y_1)(x_2,y_2)\ldots (x_{n+1},y_{n+1})\in
\tilde{\mathcal{L}}_{n,m,1}$. Then $x_{1}=1$ and $y_{1}\leq 0$. We
let
$$\Psi(P)=(x_{2},y_{2})(x_{3},y_{3})\ldots (x_{n+1},y_{n+1})(x_1,y_1).$$
This complete the proof.
\end{proof}

\begin{thm} For any $1\leq r\leq m-1$, the number of the
$(n,m)$-lattice paths $P$ such that $RML(P)=r$ is equal to the
number of the $(n,m)$-lattice paths $P=(x_1,y_1)(x_2,y_2)\ldots
(x_{n+1},y_{n+1})$ such that $RML(P)=0$ and $x_{n+1}=1$ and
independent on $r$.
\end{thm}
\begin{proof} Combining Lemmas \ref{mr>1} and \ref{mr=0}, we
immediately obtain the results as desired.
\end{proof}
\section{The pointed $(n,m)$-lattice path}
In this section, we will consider the pointed $(n,m)$-lattice paths.
We will define two parameters for an pointed $(n,m)$-lattice path:
the pointed non-positive length and the pointed rightmost minimum
length. We will give generalizations of the results in \cite{SGM}.
We will prove the Chung-Feller theorems of the pointed
$(n,m)$-lattice path on these two parameters. First, we give the
definition of the pointed $(n,m)$-lattice paths as follows.

\begin{defn} A pointed $(n,m)$-lattice paths $\dot{P}$ is a pair
$[P;j]$ such that:

(1) $P=(x_1,y_1)(x_2,y_2)\ldots (x_{n+1},y_{n+1})$ is an
$(n,m)$-lattice paths;

(2) $0\leq j\leq x_{n+1}-1$.
\end{defn}
We call the point $(m-j,0)$ the root of $P$. We use
$\mathscr{L}_{n,m}$ to denote the set of the pointed $(n,m)$-lattice
paths.
\begin{lem}\label{numberpointed} The number of the pointed $(n,m)$-lattice
paths is ${2n\choose{n}}{m\choose n+1}$.
\end{lem}
\begin{proof} Note that the number of the solutions
of the equation $\sum\limits_{i=1}^{n+1}y_i=1$ such that $1-n\leq
y_i\leq 1$ is ${2n\choose n}$. On the other hand, we let $z_i=x_i$
for all $i\in [n]$, $z_{n+1}=x_{n+1}-j$ and $z_{n+2}=j$. Since
$\sum\limits_{i=1}^{n+1}x_i=m$, $x_i\geq 1$ and $0\leq j\leq
x_{n+1}-1$, we have $\sum\limits_{i=1}^{n+2}z_i=m$, $z_i\geq 1$ for
all $i\in[n+1]$ and $z_{n+2}\geq 0$. It is easy to see that the
number of the solutions of the equation
$\sum\limits_{i=1}^{n+2}z_i=m$ such that $z_i\geq 1$ for all $i\in
[n+1]$ and $z_{n+2}\geq 0$ is ${m\choose n+1}$. Hence, the number of
the pointed $(n,m)$-lattice paths is ${2n\choose{n}}{m\choose n+1}$.
\end{proof}
\subsection{The pointed non-positive length of an pointed $(n,m)$-lattice paths}
Given a pointed $(n,m)$-lattice path $\dot{P}=[P;j]$, where
$P=(x_1,y_1)(x_2,y_2)\ldots (x_{n+1},y_{n+1})$ and $0\leq j\leq
x_{n+1}-1$, we let $PNPL(\dot{P})=NPL(P)+j$. Clearly, $0\leq
PNPL(\dot{P})\leq m-1$. We say that $PNPL(\dot{P})$ is the {\it
pointed non-positive length} of the path $\dot{P}$.

By Lemma \ref{enumeration} (1), we have the following lemma.
\begin{lem} The number of the pointed $(n,m)$-lattice paths  with pointed non-positive
length $0$ is ${m-1\choose n}c_n$.
\end{lem}

Given an $(n,m)$-lattice path $P=(x_1,y_1)(x_2,y_2)\ldots
(x_{n+1},y_{n+1})$, we let
$$P_i=(x_{i+1},y_{i+1})\ldots(x_{n+1},y_{n+1})(x_1,y_1)\ldots
(x_{i},y_{i}). $$  $P_i$ is call the {\it $i$th cyclic permutation}
of $P$. Furthermore, setting the point $(m-j,0)$ to be the root of
$P_i$, where $0\leq j\leq x_i-1$, we get a pointed $(n,m)$-lattice
paths $[P_i;j]$, denoted by $\dot{P}(i;j)$. Finally, we define a set
$\mathcal{PL}(P)$ as follows:
$$\mathcal{PL}(P)=\{\dot{P}(i;j)\mid i\in[n+1]\text{ and }0\leq j\leq x_i-1\}.$$
Clearly, we have the following lemma.
\begin{lem} $|\mathcal{PL}(P)|=m$.
\end{lem}
Recall that $<_P$ is the linear order on the set $[n+1]$. We define
a linear order $\prec_P$ on the set $\mathcal{PL}(P)$ by the
following rules:

for any $\dot{P}(i_1;j_1),\dot{P}(i_2;j_2)\in\mathcal{PL}(P)$,
$\dot{P}(i_1;j_1)\prec_P\dot{P}(i_2;j_2)$ if either (1) $i_1<_Pi_2$
or (2) $i_1=i_2$ and $j_1<j_2$.

The sequence, which is formed by the elements in the set
$\mathcal{PL}(P)$ in the increasing order with respect to $\prec_P$,
reduce a bijection from the sets  $[m]$ to $\mathcal{PL}(P)$,
denoted by $\Theta=\Theta_P$.
\begin{exa}\label{pointeddyck} Let $n=3$ and $m=5$. Let $P=(1,1)(1,-2)(1,1)(2,1)$. We draw the pointed $(3,5)$-lattice path
$\dot{P}=[P;1]$ as follows.
\begin{center}
\includegraphics[width=4cm]{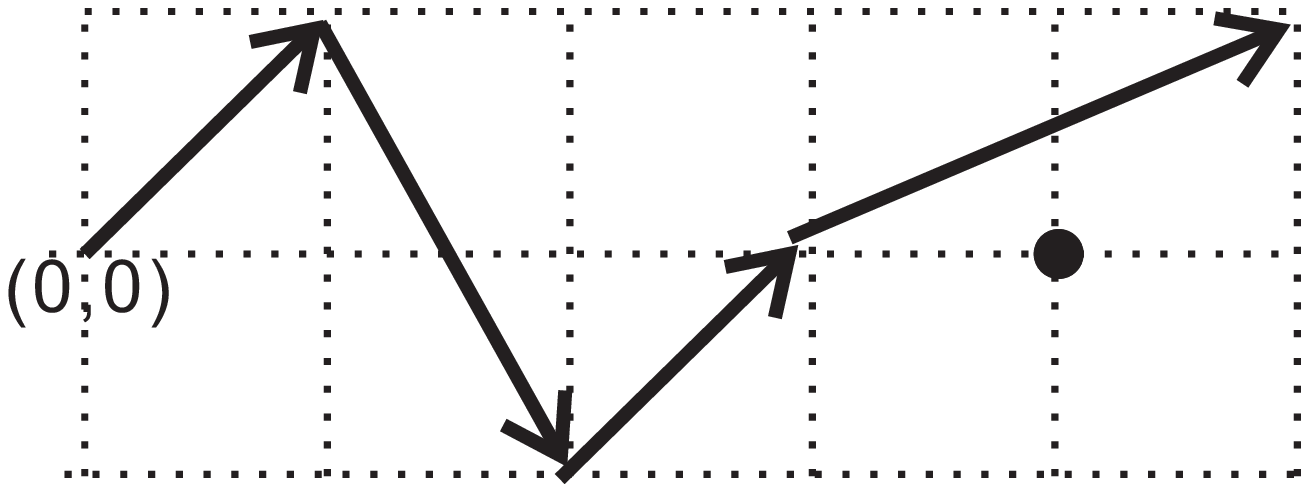}
\end{center}where the root is the point $(4,0)$ denoted by $\bullet$. Then
$PNPL(\dot{P})=3$. We write the bijection $\Theta_P$ as the
following $2\times 5$ matrix.
$$\Theta_P=\left(\begin{array}{ccccc}1&2&3&4&5\\
\dot{P}(2;0)&\dot{P}(3;0)&\dot{P}(4;0)&\dot{P}(4;1)&\dot{P}(1;0)\end{array}\right)$$
\end{exa}
\begin{thm}\label{gen} Let $P$ be an $(n,m)$-lattice path, $\mathcal{PL}(P)$
and $\Theta_P$ defined as above. Then
$$PNPL(\Theta(r))=r-1$$ for any $r\in [m]$.
\end{thm}
\begin{proof} Note that $0\leq PNPL(\Theta(r))\leq m-1$
for any $r\in[m]$. It is sufficient to prove that
$PNPL(\Theta(r+1))=PNPL(\Theta(r))+1$ for any $r\in [m-2]$. Suppose
$$P=(x_1,y_1)(x_2,y_2)\ldots (x_{n+1},y_{n+1})$$ and
$\Theta(r)=\dot{P}(s;t)\in\mathcal{PL}(P)$. Let $\pi_P$ be the
sequence formed by writing $[n+1]$ in the increasing order with
respect to $<_P$  and $\pi_P^{-1}(s)=k$. Then
$PNPL(\Theta(r))=\sum\limits_{j=1}^{k-1}x_{\pi_P(j)}+t$. Now,
suppose $\Theta(r+1)=\dot{P}(\tilde{s};\tilde{t})$. We discuss the
following two cases:

{\it Case I.} $s=\tilde{s}$

Then $\tilde{t}=t+1$. This implies
$PNPL(\Theta(r+1))=PNPL(\Theta(r))+1$.

{\it Case II.} $s<_P \tilde{s}$

Then $\pi_P(k+1)=\tilde{s}$, $t=x_s-1$ and $\tilde{t}=0$. Thus,
$$PNPL(\Theta(r+1))=\sum\limits_{j=1}^{k}x_{\pi_P(j)}=\sum\limits_{j=1}^{k-1}x_{\pi_P(j)}+x_s=PNPL(\Theta(r))+1.$$
This complete the proof.\end{proof}
\begin{exa} We consider the path $P$ in Example \ref{pointeddyck}. We draw the
pointed lattice path $\Theta(r)$ as follows:
\begin{center}
\includegraphics[width=12cm]{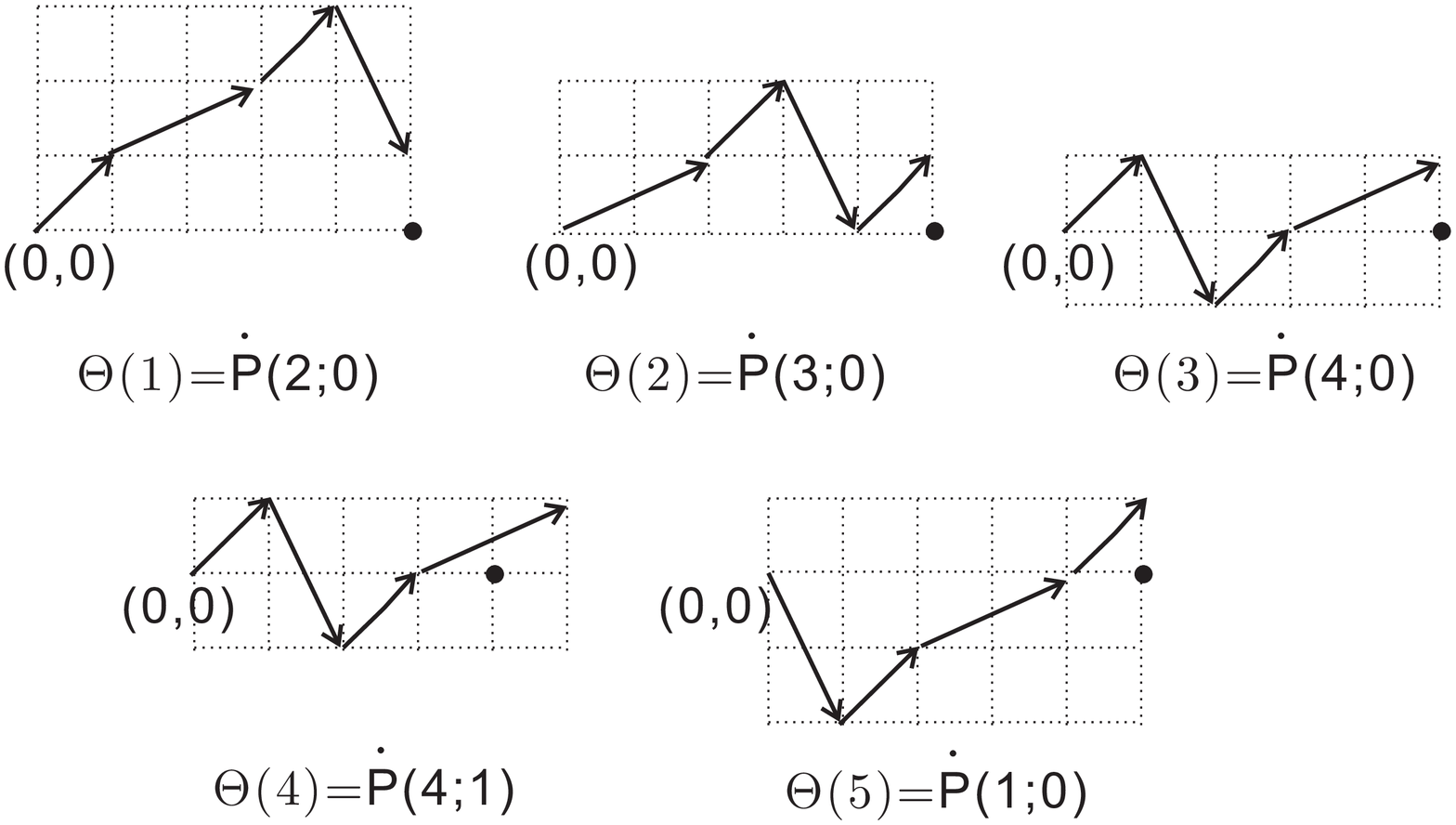}
\end{center}
\end{exa}
\begin{rem} Let $\dot{P}=[P;j]$ be a pointed $(n,m)$-lattice path, where $P=(x_1,y_1)\ldots
(x_{n+1},y_{n+1})$ and $0\leq j\leq x_{n+1}-1$. Setting $m=n+1$, we
have $x_i=1$ for all $i$ and $j=0$. Let $Y=(y_1,\ldots,y_{n+1})$.
Then ${E}(Y)=PNPL(\dot{P})$. This tells us that Lemma
\ref{bookchung} can be viewed as a corollary of Theorem \ref{gen}.
\end{rem}

We use $\mathscr{L}_{n,m,r}$ to denote the set of the pointed
$(n,m)$-lattice paths with pointed non-positive length $r$. Clearly,
$\mathscr{L}_{n,m}=\bigcup\limits_{r=0}^{m-1}\mathscr{L}_{n,m,r}$.
Let $l_{n,m,r}=|\mathscr{L}_{n,m,r}|$.

\begin{cor}\label{chungdyck} For any $0\leq r\leq m-1$, the number of the pointed $(n,m)$-lattice paths with pointed non-positive length
$r$ is equal to the number of the pointed $(n,m)$-lattice paths with
pointed non-positive length $0$ and independent on $r$, i.e.,
$l_{n,m,r}=\frac{1}{m}{2n\choose n}{m\choose{n+1}}$.
\end{cor}
\begin{proof} First, we define an equivalent relation on the set
$\mathscr{L}_{n,m}$. Let $\dot{P}=[P;i]$ and $\dot{Q}=[Q;j]$ be two
pointed $(n,m)$-lattice paths. Suppose $P=(x_1,y_1)\ldots
(x_{n+1},y_{n+1})$. Recall $P_k$ denote the $k$th cyclic permutation
of $P$, i.e.,
$P_k=(x_{k+1},y_{k+1})\ldots(x_{n+1},y_{n+1})(x_1,y_1)\ldots
(x_{k},y_{k}). $ We say $\dot{Q}$ and $\dot{P}$ is equivalent,
denoted by $\dot{Q}\sim\dot{P}$, if $Q=P_k$ for some $k\in[n+1]$.
Hence, given a pointed lattice path $\dot{P}\in\mathscr{L}_{n,m}$,
we define a set $EQ(\dot{P})$ as
$EQ(\dot{P})=\{\dot{Q}\in\mathscr{L}_{n,m}\mid
\dot{Q}\sim\dot{P}\}$. We say that the set $EQ(\dot{P})$ is an
equivalent class of the set $\mathscr{L}_{n,m}$. Clearly,
$|EQ(\dot{P})|=m$. Now, we may suppose that the set
$\mathscr{L}_{n,m}$ has $t$ equivalent class. Then
$t=\frac{1}{m}{2n\choose n}{m\choose{n+1}}$. For any $0\leq r\leq
m-1$, from Theorem \ref{gen}, every equivalent class contains
exactly one element with pointed non-positive length $r$. Hence,
$l_{n,m,r}=t=\frac{1}{m}{2n\choose n}{m\choose{n+1}}$.
\end{proof}

\subsection{The pointed rightmost minimum length of an pointed $(n,m)$-lattice paths}
Let $\dot{P}=[P;j]$ be a pointed $(n,m)$-lattice path, where
$P=(x_1,y_1)(x_2,y_2)\ldots (x_{n+1},y_{n+1})$ is a $(n,m)$-lattice
path and $0\leq j\leq x_{n+1}-1$. Recall that $RML(P)$ is the
rightmost minimum length of $P$. We let $PRML(\dot{P})=RML(P)+j$ and
call $PRML(\dot{P})$ the {\it pointed rightmost minimum length} of
$\dot{P}$.

Note that $PNPL(P)=0$ if and only if $PRML(P)=0$ for any pointed
$(n,m)$-lattice path. We immediately obtain the following lemma.

\begin{lem} The number of the pointed $(n,m)$-lattice paths  with pointed rightmost minimum length $0$ is ${m-1\choose n}c_n$.
\end{lem}

First, given a $(n,m)$-lattice path $P$, we recall that $\pi_P$ is
the sequence formed by writing $[n+1]$ in the increasing order with
respect to $<_P$. Suppose $\pi_P(1)=i$. Let
$\sigma_P=(\sigma_P(1),\sigma_P(2),\ldots,\sigma_P(n+1))=(i,i-1,\ldots,1,n+1,n,\ldots,i+1)$.

Using $\sigma_P$, we define a new linear order $\prec^*_P$ on the
set $\mathcal{PL}(P)=\{\dot{P}(i;j)\mid i\in[n+1]\text{ and }0\leq
j\leq x_i-1\}$ by the following rules:

for any $\dot{P}(i_1;j_1),\dot{P}(i_2;j_2)\in\mathcal{PL}(P)$,
$\dot{P}(i_1;j_1)\prec^*_P\dot{P}(i_2;j_2)$ if either (1)
$\sigma^{-1}_P(i_1)<\sigma^{-1}_P(i_2)$ or (2) $i_1=i_2$ and
$j_1<j_2$.

The sequence, which is formed by the elements in the set
$\mathcal{PL}(P)$ in the increasing order with respect to
$\prec^*_P$, reduce a bijection from the sets $[m]$ to
$\mathcal{PL}(P)$, denoted by $\Gamma=\Gamma_P$.
\begin{exa}  Consider the path $P$ and the pointed path $\dot{P}$ in Example \ref{pointeddyck}.
we have $PRML(\dot{P})=3$. It is easy to see $\sigma_P=(2,1,4,3)$.
We write the bijection $\Gamma_P$ as the following $2\times 5$
matrix.
$$\Gamma_P=\left(\begin{array}{ccccc}1&2&3&4&5\\
\dot{P}(2;0)&\dot{P}(1;0)&\dot{P}(4;0)&\dot{P}(4;1)&\dot{P}(3;0)\end{array}\right)$$
\end{exa}
\begin{thm}\label{mgen} Let $P$ be an $(n,m)$-lattice path and
$\Gamma$ defined as above. Then
$$PRML(\Gamma(r))=r-1$$ for any $r\in [m]$.
\end{thm}

\begin{proof} It is sufficient to prove that
$PRML(\Gamma(r+1))=PRML(\Gamma(r))+1$. Suppose
$\Gamma(r)=\dot{P}(i_1;j_1)$ and $\Gamma(r+1)=\dot{P}(i_2;j_2)$. If
$i_1=i_2$, then $j_1+1=j_2$. Clearly,
$PRML(\Gamma(r+1))=PRML(\Gamma(r))+1$. We consider the case with
$\sigma^{-1}_P(i_1)<\sigma^{-1}_P(i_2)$. Let $k=\sigma^{-1}_P(i_1)$.
Then  $\sigma^{-1}_P(i_2)=k+1$, $j_1=x_{i_1}-1$ and $j_2=0$. We have
$PRML(\dot{P}(i_2;j_2))=\sum\limits_{j=1}^{k}x_{\sigma_P(j)}=\sum\limits_{j=1}^{k-1}x_{\sigma_P(j)}+x_{i_1}=PRML(\dot{P}(i_1;j_1))+1$.
\end{proof}
\begin{exa}  We consider the path $P$ in Example \ref{pointeddyck}. We draw the
pointed lattice path $\Gamma(r)$ as follows:
\begin{center}
\includegraphics[width=12cm]{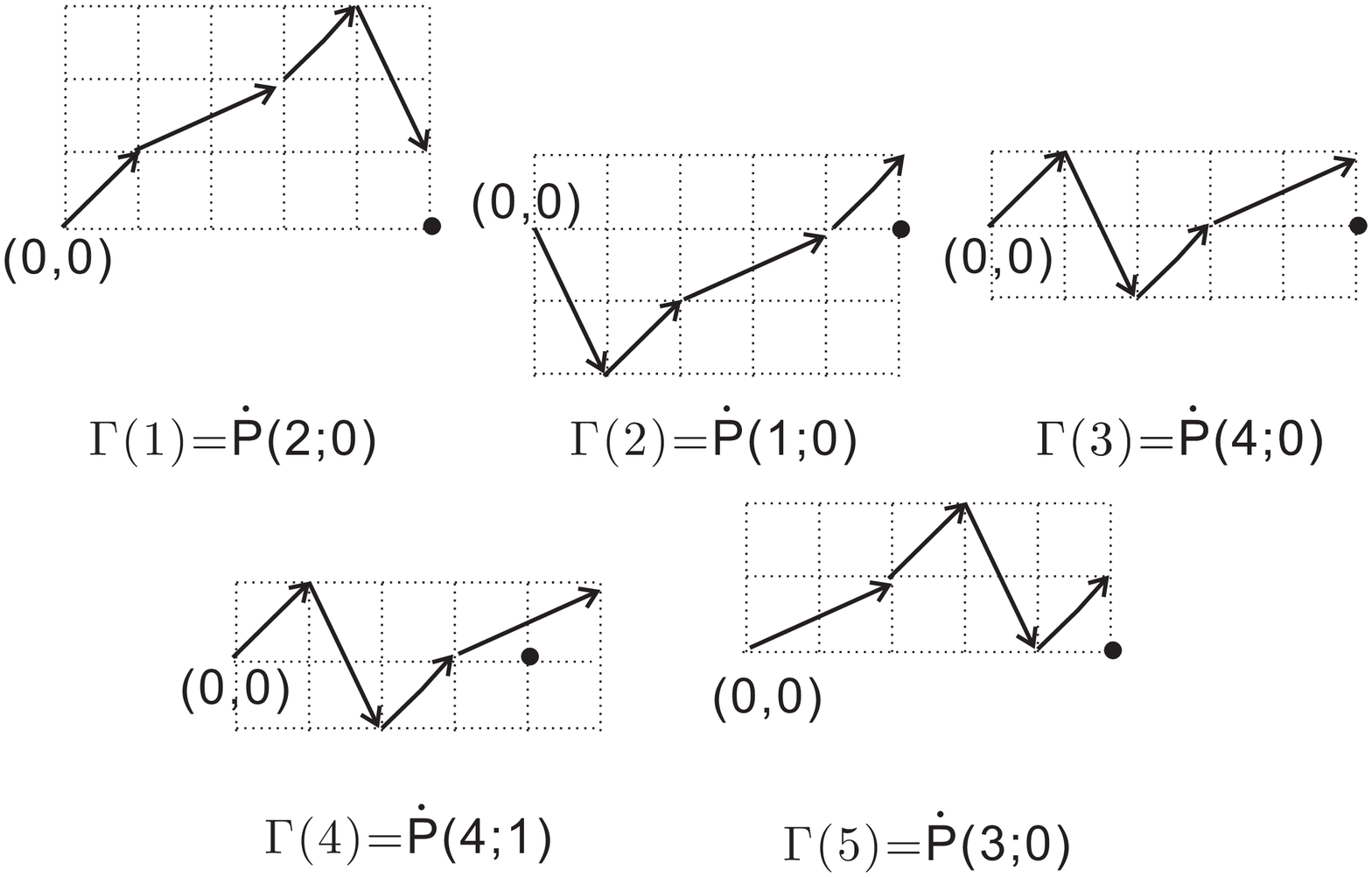}
\end{center}
\end{exa}

We use $\mathscr{M}_{n,m,r}$ to denote the set of the pointed
$(n,m)$-lattice paths with pointed rightmost minimum length $r$.
Clearly,
$\mathscr{L}_{n,m}=\bigcup\limits_{r=0}^{m-1}\mathscr{M}_{n,m,r}$.
Let $d_{n,m,r}=|\mathscr{M}_{n,m,r}|$.

\begin{cor} For any $0\leq r\leq m-1$, the number of the pointed $(n,m)$-lattice paths with pointed rightmost minimum length
$r$ is equal to the number of the pointed $(n,m)$-lattice paths with
pointed rightmost minimum length $0$ and independent on $r$, i.e.,
$d_{n,m,r}=\frac{1}{m}{2n\choose n}{m\choose{n+1}}$.
\end{cor}
\begin{proof} Similar to the proof of Corollary \ref{chungdyck}, we
can obtain the results as desired.
\end{proof}

\section{The application of the main theorem}

In fact, by Theorems \ref{gen} and  \ref{mgen}, we may find the
Chung-Feller theorems of many different  $(n,m)$-lattice paths on
the parameter: the pointed non-positive length and the pointed
rightmost minimum length. For example, we let $A$ and $B$ be two
finite subsets of the set $\mathbb{P}$. Let
$\mathcal{S}=\mathcal{S}_{A}\cup\mathcal{S}_{B}\cup\{(1,1)\}$, where
$\mathcal{S}_{A}=\{(2i-1,-1)\mid i\in A\}$ and
$\mathcal{S}_{B}=\{(2i,0)\mid i\in B\}.$ In \cite{Ma}, we have
proved the following corollary by the generating function methods.
Using Theorems \ref{gen} and  \ref{mgen}, we can  reobtain the
corollary.
\begin{cor}\label{resultsreprove1} Let $\mathscr{P}_{n,m}$ be the set of
the pointed lattice paths in the plane $\mathbb{Z}\times \mathbb{Z}$
which (1) only use steps in the set $\mathcal{S}$; (2) have $n+1$
steps; (3) go from the origin to the point (m,1). Then in
$\mathscr{P}_{n,m}$,
\item{(1) the number of the pointed lattice paths
with pointed non-positive length $r$ is equal to the number of the
pointed lattice paths with pointed non-positive length $0$ and
independent on $r$;\\
(2) the number of the pointed lattice paths with pointed rightmost
minimum length $r$ is equal to the number of the pointed lattice
paths with pointed rightmost minimum length $0$ and independent on
$r$.}
\end{cor}
\begin{proof} (1) It is easy to see that a pointed lattice path $P$ in $\mathscr{P}_{n,m}$ can be view as a
pointed $(n,m)$-lattice path $(x_1,y_1)\ldots (x_{n+1},y_{n+1})$
such that $(x_i,y_i)\in \mathcal{S}$ for all $i\in[n+1]$. By Theorem
\ref{gen}, using a similar method as Corollary \ref{chungdyck}, we
get the results as desired.\\
(2) The proof is omitted.
\end{proof}

\subsection *{Acknowledgements}

\ \ \ \ The authors would like to thank Professor Christian
Krattenthaler for his valuable suggestions.

 \end{document}